\documentclass[11pt]{amsart}
\usepackage{geometry}                
\usepackage{graphicx}
\usepackage{amssymb, amsmath,latexsym}
\usepackage{epstopdf,xypic}

\newtheorem{theorem}{Theorem}
\newtheorem{definition}{Definition}
\newtheorem{proposition}{Proposition}
\newtheorem{example}{Example}
\newtheorem{conjecture}{Conjecture}
\newtheorem{corollary}{Corollary}
\newcommand{\lsem}{[\![}
\newcommand{\rsem}{]\!]}

\title{Boolean formulae, hypergraphs and combinatorial topology}
\author{James Conant}
\author{Oliver Thistlethwaite}
\date{}                                           

\begin{document}
\begin{abstract}
With a view toward studying the homotopy type of spaces of Boolean formulae, we introduce a simplicial complex, called the theta complex, associated to any hypergraph, which is the Alexander dual of the more well-known independence complex.
 In particular, the set of satisfiable formulae in $k$-conjunctive normal form with $\leq n$ variables has the homotopy type of $\Theta(\operatorname{Cube}(n,n-k))$, where $\operatorname{Cube}(n,n-k)$ is a hypergraph associated to the $(n-k)$-skeleton of an $n$-cube. We make partial progress in calculating the homotopy type of theta for these cubical hypergraphs, and we also give calculations and examples for other hypergraphs as well. Indeed studying the theta complex of hypergraphs is an interesting problem in its own right.
\end{abstract}
\maketitle

\section{Introduction}
In this paper we introduce and study a new concept in combinatorial topology, which we call the theta complex of a hypergraph. A hypergraph, $\mathcal H$, is a set of vertices and a set of subsets of the vertices, called hyperedges. The theta complex $\Theta(\mathcal H)$ is a simplicial complex with simplices spanned by vertices that are in the complement of at least one hyperedge. Despite the simplicity of this definition, the homotopy type of $\mathcal H$ is usually not obvious even for simple  hypergraphs.

Our main interest in defining and pursuing this construction is the hope that topology can be brought to bear on the famous P/NP question of computer science. Very briefly, a \emph{decision problem} is a function from a set of input strings to the set $\{\text{Yes},\text{No}\}$. A decision problem is said to be a P  problem if there is an algorithm (implemented on a Turing machine) which terminates in the correct answer of ``yes" or ``no" after a number of steps bounded by a polynomial in the size of the input string. 
On the other hand, an NP problem is a decision problem that can be ``checked" in polynomial time, and an NP complete problem is an NP problem to which every other NP problem can be reduced in polynomial time. The class of P problems is a subset of the class of NP problems, but it is widely believed that they are not equal. I.e. there is no polynomial time algorithm for solving an NP-complete problem.


An important class of decision problems is the class of $k$-SAT problems, which ask whether a Boolean formula of a given type is satisfiable (i.e. is not a contradiction.) The $k$-SAT problem restricts to formulae which are conjunctions of disjunctions of $k$ literals. These are NP problems because 
an assignment of truth values to the variables can be verified to be a satisfaction in polynomial time. It turns out that $2$-SAT is a P problem, but $k$-SAT for $k\geq 3$ is an NP complete problem. (This is Cook's Theorem.)
Thus one attempt to understand the P/NP question is to understand the difference between $2$-SAT and $3$-SAT. (See \cite{F1,F3}.)

One can assign a simplicial complex to any set of Boolean formulae by letting there be a simplex for every chain of implications 
$$\phi_0\Rightarrow \phi_1\Rightarrow\cdots\Rightarrow \phi_k.$$ 
If the set of formulae contains a contradiction or a tautology then the simplicial complex is a cone, and hence contractible. In the case of $k$-SAT, there are plenty of contradictions but no tautologies, so 
the simplicial complex of satisfiable formula has a chance to be topologically interesting.
One may hope that information about the topology or metric structure of such spaces can be used to distinguish P and NP.  Unfortunately, taking this simplicial realization for $k$-SAT seems to yield a contractible space when one uses an infinite number of variables, although the large-scale metric structure of this space deserves further study. (See \cite{F2}, which proposes that the study of large scale geometry of spaces associated to decision problems via ultrafilter limits could be used to distinguish P from NP.) In this paper, the approach of restricting to a finite number of variables is taken. Indeed, let $|k\text{-SAT-}n|$ be the simplicial complex of satisfiable formulae in $n$-variables in $k$-conjunctive normal form. Then the relevance  of the theta complex becomes apparent (Theorem \ref{cubenerve}): $$|k\text{-SAT-}n|\simeq \Theta(\operatorname{Cube}(n,n-k)),$$ where $\operatorname{Cube}(n,\ell)$ is the hypergraph whose vertices are the vertices of an $n$-cube,  and whose hyperedges come from the $\ell$-dimensional faces of the $n$-cube. 

So the problem now becomes to analyze the homotopy type of $\Theta(\operatorname{Cube}(n,\ell))$. This appears to be a difficult problem, the partial analysis of which forms the core of this paper. 
 Looking at the low dimensional data, one can conjecture a formula for $\Theta(\operatorname{Cube}(n,n-2))$, the case of $2$-SAT. Namely Conjecture \ref{factorialconjecture}, due to Oliver Thistlethwaite \cite{T}, states
$$\Theta(\operatorname{Cube}(n,n-2))\simeq \vee_{(2n-3)!!}S^{2n-2}.$$
It is surprising that the proof of this has been so elusive. In section~\ref{groupsection} we at least verify that this conjecture gives the correct Euler characteristic modulo $p$ for all $n\geq p$. 
On the other hand, the pattern for $k$-SAT for $k\geq 3$ remains hidden, but we can at least say they are not in general wedges of same-dimensional spheres. Indeed, this could be the topological difference between $2$-SAT and $k$-SAT for $k\geq 3$. (Conjecture \ref{spread})

That said, this paper is a preliminary investigation and does not address whether these topological phenomena are merely accidents or are related to the computational complexity of the corresponding decision questions. However, we believe that these topological phenomena are interesting in their own right independently of whether or not they do turn out to play a role in the P/NP question.

The main tool used in the paper is the technique of discrete vector fields \cite{Form1,Form2}, which are an efficient tool for calculating the homotopy types of finite simplicial complexes. In section \ref{vecsection} we give a brief overview of the technique. In section \ref{calcsection} we use this technique to calculate examples of $\Theta(\mathcal H),$ including $\Theta(\operatorname{Cube}(3,1))\simeq S^4\vee S^4\vee S^4$ (Example \ref{cubeexample}), and we also present the results of computer calculations for the case of cubes (Theorems \ref{cube1} and \ref{cube2}). 

Finally, in section \ref{groupsection}, we consider $p$-group actions on hypergraphs. A nice feature of the theta complex is that it behaves well with respect to such actions. Namely, Theorem \ref{grouptheorem} states that if $G$ is a finite $p$-group $$\chi(\Theta(\mathcal H))\cong \chi(\Theta(\mathcal H/G)) \mod p.$$ After giving a couple of examples we prove Theorem \ref{groupcube} which states that the Euler characteristic of $\Theta(\operatorname{Cube}(n,n-2))$ matches Conjecture \ref{factorialconjecture} modulo $p$, for all primes $p\leq n$. In fact, using discrete vector fields, we show the much stronger statement that  
$\Theta(\operatorname{Cube}(n,n-2)/\mathbb Z_p)$ is contractible whenever  $n\geq p$.
 
We have already intimated that the study of $\Theta(\mathcal H)$ is interesting in its own right, and in particular the case when $\mathcal H$ is a graph is an interesting subcase. Indeed the $1$-skeleta of $n$-dimensional cubes yields the puzzling sequence of Euler characteristics
$$0,4,8,12,144,7716,\ldots.$$
The class of graphs is studied in \cite{students} by students in an REU project.  In the last section we observe what an existing connectivity estimate \cite{e1} gives for the case of cubes. 

\hspace{1em}

\noindent{\bf Acknowledgements:} This research was partially supported by NSF Grant DMS 0604351. 
Thanks to Katie Bolus, Joshua Edmonds, Sara Evans, Tony Zamberlan \cite{students}, Nikolai Brodskiy, Mike Freedman, Jakob Jonsson, Alexander Engstr\"om, and the anonymous referee for helpful discussions.

\section{Basic Definitions}
\begin{definition}
A hypergraph, $\mathcal H$, is a pair $(V,H)$ where $V$ is a nonempty set, whose elements are called vertices and where $H$ is a collection of subsets of $V$. The elements of $H$ are called hyperedges.
\end{definition}

Note that a graph is a type of hypergraph where each hyperedge contains exactly two vertices.
There are a couple of basic operations one can do to hypergraphs to form new hypergraphs.

\begin{definition} Let $\mathcal H=(V,H)$ be a hypergraph.
\begin{enumerate}
\item The dual hypergraph $\mathcal H^*$ has vertex set equal to $H$, and has hyperedges corresponding to elements of $V.$ Namely a dual hyperedge associated to a vertex $v$ is defined to consist of all hyperedges containing $v$. 
\item The simplicial complex $\Theta(\mathcal H)$ is defined so that simplices are spanned by all finite subsets of complements of hyperedges of $\mathcal H$. 
\end{enumerate}
\end{definition}

\noindent{\bf Remark:} In this paper we will not distinguish between a combinatorial simplicial complex and 
its geometric realization.

This definition is related to one which has already been extensively studied in combinatorial topology. (See \cite{e1}).

\begin{definition}
Let $\mathcal H=(V,H)$ be a hypergraph. The \emph{independence complex}, $I(\mathcal H)$ is defined
 to have simplices which consist of collections of vertices from $V$, such that no set of vertices spans a hyperedge.\end{definition}

We also recall the definition of the Alexander Dual of a complex.  (See \cite{jonsson}.)

\begin{definition}
Let $X$ be a simplicial complex with vertex set $V$. The Alexander Dual $AD(X)$ has simplices $\sigma\subset V$ whenever $V\setminus\sigma$ is not a simplex of $X$.
\end{definition}

\begin{proposition}
We have that $\theta(\mathcal H)\cong AD(I(\mathcal H))$.
\end{proposition}
\begin{proof}
A simplex is in $I(\mathcal H)$ if it contains no hyperedge. Hence a simplex does not lie in $I(\mathcal H)$ if it contains a hyperedge, and the complement then omits at least one hyperedge. 
\end{proof}

The Alexander Dual complex exhibits a duality between homology and cohomology \cite{jonsson}.

\begin{theorem}
There is an isomorphism
$$\widetilde{H}^d(X)\cong \widetilde{H}_{|V|-d-3}(AD(X))$$
\end{theorem}

Getting back to theta complexes, we prove a basic theorem.

\begin{theorem}
Suppose that $H$ is a finite hypergraph.
Then $\Theta(\mathcal H^*)\simeq \Theta(\mathcal H)$. (Here $\simeq$ denotes the equivalence relation of homotopy equivalence.)\end{theorem}
\begin{proof}
Let the vertex of $\mathcal H^*$ corresponding to the hyperedge $h$ be denoted $v_h$ and let the hyperedge of $\mathcal H^*$ corresponding to the vertex $v$ be denoted $h_v$.

We use the theorem that the nerve of an open cover of a paracompact space such that all finite intersections are contractible or empty (a good cover) is homotopy equivalent to the original space. (\cite{H} Corollary 4G.3 p459). Cover $\Theta(\mathcal H)$ by open sets $U_h$ for each hyperedge $h$, defined to be small neighborhoods of the simplices represented by complements of the hyperedges $h$. Then this is a good cover. (It is a cover by the hypothesis that every vertex avoids at least one hyperedge.) So,  the nerve complex $\mathcal N$ has a vertex $v_h$ for each hyperedge $h$ of $\mathcal H$.
An intersection of the sets $U_{h_1}\cap U_{h_2}\cap\cdots\cap U_{h_k}$ is nonempty iff the corresponding simplices have at least one vertex in common, which is to say there is some vertex $v$ of $\mathcal H$ such that $v\not\in h_i$ for all $i$. So
\begin{align*}
&[v_{h_1},\ldots,v_{h_k}] \text{ is a simplex of } \mathcal N \Leftrightarrow\\
&\text{There is some }v\text{ such that }v\not\in h_i\text{ for any } i\Leftrightarrow\\
&\text{There is some }v\text{ such that }\{v_{h_1},\ldots, v_{h_k}\}\subset h_v^c\Leftrightarrow \\
&[v_{h_1},\ldots,v_{h_k}] \text{ is a simplex of } \Theta(\mathcal H^*)
\end{align*}
\end{proof}

One may wonder whether disconnected hypergraphs can be analyzed in terms of their components. The following proposition offers an affirmative answer.

\begin{proposition}\label{whichone}
Consider the disjoint union of hypergraphs $\mathcal H_1\coprod \mathcal H_2$. Then
$$\Theta\left(\mathcal H_1\coprod \mathcal H_2\right)\simeq \Sigma(\Theta(\mathcal H_1)*\Theta\left(\mathcal H_2)\right).$$(Here $\Sigma$ represents suspension and $*$ represent the join.)
\end{proposition} 
\begin{proof}
Let $B_i$ be the simplex spanned by the vertex set of $\mathcal H_i$. Then $\Theta(\mathcal H_i)\subset B_i$.
 In order to be a simplex in $\Theta
(\mathcal H_1\coprod \mathcal H_2)$ you can either miss an edge in $\mathcal H_1$ or one in $\mathcal H_2$. Thus, $\Theta(\mathcal H_1\coprod \mathcal H_2)=(B_1*\Theta(\mathcal H_2))\cup (\Theta(\mathcal H_1)*B_2)\subset B_1*B_2$. The proposition now follows from the following general statement: if $K_i\subset B_i$ is an inclusion of cell complexes, with $B_i$ contractible, then $(K_1*B_2)\cup (B_1*K_2)\subset B_1*B_2$ is homotopy equivalent to $\Sigma(K_1*K_2)$. 
When $B_i=C(K_i)$, where ``$C$" denotes the cone of a space, we exactly get $\Sigma(K_1*K_2)$, since $C(K_1)*K_2=C(K_1*K_2)=K_1*C(K_2)$, and in fact we can reduce to this case by showing that the pair 
$(B_i,K_i)$ is homotopy equivalent to $(C(K_i),K_i)$ rel $K_i$. Clearly $(B_i\times\{0\},K_i\times\{0\})\simeq (B_i\times[0,1],K_i\times\{0\})$ rel $K_i\times\{0\}$. This is homotopy equivalent to $((K_i\times[0,1])\cup (B_i\times\{1\}),K_i\times\{0\})$ rel $K_i\times\{0\}$. Finally contracting $B_i\times\{1\}$ yields the desired result.
\end{proof}

\begin{corollary}
Suppose a hypergraph $\mathcal H$ has an isolated vertex. (That is no hyperedge contains it.) Then $\Theta(\mathcal H)$ is contractible.
\end{corollary}

To finish this section, we record the fact that the class of theta complexes includes all simplicial complexes.
\begin{proposition}
Let $K$ be a simplicial complex. Then there is a hypergraph $\mathcal H$ such that $\Theta(\mathcal H)=K$.
\end{proposition}
\begin{proof}
Let $\mathcal H$ have the same vertex set as $K$ and for every simplex of $K$ let the complement of the vertices spanning it be a hyperedge.
\end{proof}
\section{Boolean Formulae}
A Boolean formula is a well-formed formula constructed from variables $x_1,\ldots,x_n$ and the basic logical operations of $\vee$ (OR), $\wedge$ (AND), and $\neg$ (NOT). Negation of a variable is also denoted with an overbar. 

\begin{definition}
\begin{enumerate}
\item The formula $\phi_1\vee \phi_2\vee\cdots\vee \phi_k$ is said to be the disjunction of the formulas $\phi_i$.
\item The formula $\phi_1\wedge \phi_2\wedge \cdots\wedge \phi_k$ is said to be the conjunction of the formulas $\phi_i$.
\item A \emph{literal} is a variable, $x_i$, or its negation, $\overline{x_i}$. 
\item A formula is in conjunctive normal form if it is a conjunction of clauses where each clause is a disjunction of literals, no clauses are duplicated, and the same variable does not appear twice in any clause.
\item A formula is in $\ell$-conjunctive normal form if it is in conjunctive normal form where every clause contains $\ell$ literals. 
\end{enumerate}
\end{definition}

The importance of the class of $\ell$-conjunctive formulas, as mentioned in the introduction, is indicated by the fact that checking the satisfiability of a $2$-conjunctive formula is a P problem (called 2-SAT), whereas checking the satisfiability of a $3$-conjunctive formula is an NP complete problem (called 3-SAT). 

\begin{definition}
Let $\ell\text{-}SAT\text{-}n$ denote the set of satisfiable $\ell$-conjunctive formulas in the variables $x_1,\ldots,x_n$.  Define  $|\ell\text{-}SAT\text{-}n|$, the geometric realization, to be the simplicial complex with vertex set equal to $\ell\text{-}SAT\text{-}n$, and a $k$-simplex $[\phi_0,\ldots,\phi_k]$ whenever we have the chain of implications $$\phi_0\Rightarrow\phi_1\Rightarrow \cdots\Rightarrow \phi_k.$$
\end{definition}

\noindent{\bf Remark:} This definition mimics the definition of the geometric realization of a poset. The set $\ell\text{-}SAT\text{-}n$ is not actually a poset under $\Rightarrow$ because there are logically equivalent but distinct formulae. For example $(x_1\vee x_2)\wedge (\bar{x}_1\vee x_2)$ is equivalent to $(x_1\vee x_2)\wedge (\bar{x}_1\vee x_2)\wedge (x_3\vee x_2)$. 

\begin{definition}
Let $\operatorname{Cube}(n,k)$ be the hypergraph whose vertices are the vertices of the $n$-cube and whose hyperedges are the sets of vertices spanning $k$-dimensional faces of the $n$-cube.
\end{definition}

\begin{theorem}
\label{cubenerve}
There is a homotopy equivalence $$|\ell\text{-}SAT\text{-}n|\simeq \Theta(\operatorname{Cube}(n,n-\ell))$$
\end{theorem}
\begin{proof}
Fix an assignment, $\tau$, of ``T" or ``F" to each variable $x_1,\ldots,x_n$. Form an open cover $\{U_\tau\}$ of $|\ell\text{-}SAT\text{-}n|$ as follows. $U_\tau$ is a small neighborhood of the union of simplices $[\phi_0,\cdots,\phi_k]$ where $\tau$ is a satisfaction for each formula $\phi_i$ in the simplex. We claim that any nonempty intersection of these is contractible. Consider the set of formulae which are vertices in $\cap_{i} U_{\tau_i}$. Take the conjunction of all these formulae, removing duplicate clauses. This is still satisfied by each $\tau_i$, and furthermore implies every formula in the intersection. Thus the intersection is a cone on this formula. So the $U_\tau$'s form a good cover. We consider the nerve of this cover. The vertices correspond to truth assignments $\tau$ and these are in $1-1$ correspondence with vertices of the $n$-cube. Now let us consider which collections of $U_\tau$ have nontrivial intersection.
Note that the clause $x_{i_1}\vee\cdots\vee x_{i_k}$ is satisfiable away from the $(n-k)$-face of the cube $x_{i_1}=F,x_{i_2}=F,\ldots,x_{i_k}=F$, and similarly for negated variables. So each clause is satisfiable in the complement of an $(n-k)$-face of the cube. So if $\{\tau_1,\ldots,\tau_m\}$ avoids an $(n-k)$-face, the intersection $\cap_i U_{\tau_i}$ is nonempty, since the clause corresponding to that face is in the intersection. Similarly, if $\{\tau_1,\ldots,\tau_m\}$ hits every $(n-k)$-face, then a formula in the intersection $\cap_i U_{\tau_i}$ could not contain any clause, meaning that the intersection is actually empty. 
\end{proof}

\section{Discrete Vector Fields}\label{vecsection}

Let $K$ be a finite simplicial complex. A vector is defined to be a pair of simplices ${(\sigma,\tau)}$ such that $\sigma$ is a codimension $1$ face of $\tau$. A vector field, by definition, is a collection of vectors so that no simplex appears in more than one vector. The critical simplices, by definition, are those that do not appear in any vector. A gradient path with respect to a given vector field is a sequence of simplices
$$\sigma_1,\tau_1,\sigma_2,\tau_2,\ldots,\sigma_k,\tau_k$$
such that each $(\sigma_i,\tau_i)$ is a vector, and $\sigma_{i+1}$ is a codimension $1$ face of $\tau_i$ distinct from $\sigma_i$. A vector field is said to be a gradient field if no gradient path is a loop. The importance of this definition is the following result \cite{Form1,Form2}.

\begin{theorem}\label{stuff}
If $K$ is a simplicial complex with a gradient  field, then it is homotopy equivalent to a cell complex with one $i$-cell for every critical $i$-simplex. \end{theorem}

Given a simplicial complex, $K$, choose a sequence of distinct vertices $v_1,\ldots,v_n$. This gives rise to a vector field $\mathbf D_{v_1,\ldots,v_n}$ defined recursively in the following way. 
Let $\mathbf D^1=\{(\sigma,\sigma\cup\{v_1\})\}$ where $\sigma$ ranges over all simplices not containing $v_1$ which are in $K$ and such that $\sigma\cup\{v_1\}$ is also in $K$. Let $C^1$ be the set of critical simplices of this vector field. Now, given $\mathbf D^i$ and $C^i$ define the vector field 
$$\mathbf D^{i+1}=\mathbf D^i\cup\{(\sigma,\sigma\cup\{v_{i+1}\}):v_{i+1}\not\in\sigma\in C^i,\sigma\cup\{v_{i+1}\}\in C^i\},$$
and let $C^{i+1}$ be the critical simplices of this vector field. Finally $\mathbf D_{v_1,\ldots,v_n}:=\mathbf D^n$.

A vector field of this form is called \emph{sequential}.

The following proposition is frequently a time-saver.
\begin{proposition}
A sequential vector field $\mathbf D_{v_1,\ldots, v_n}$ is always gradient.
\end{proposition}
\begin{proof}
Suppose we have a gradient loop.
Let $k$ be the minimal number such that a vector $(\sigma,\sigma\cup\{v_k\})$ appears in the gradient loop. Since we have a loop, at some point the vertex $v_k$ will have to be removed when passing from some $\tau_i$ to $\sigma_{i+1}$.  Now by minimality of $k$, we must have $\tau_{i+1}=\sigma_{i+1}\cup \{v_\ell\}$ for $\ell\geq k$. So we have that $\sigma_{i+1}\in C^{\ell-1}\subset C^{k-1}$.
But then $(\sigma_{i+1},\sigma_{i+1}\cup\{v_k\})$ is a vector in $\mathbf D^k$. So $v_k=v_\ell$, which is a contradiction.
\end{proof}

\section{Calculations and Conjectures}\label{calcsection}

\subsection{Graphs}
Graphs are among the most tractable hypergraphs to analyze. Hence we start with some calculations in this context to give the reader a feel for how vector fields work. The computations in this section are well-known for the Alexander dual independence complexes. See, for example \cite{kozlov}.

\begin{example}Let $I_n$ denote the graph which is $n$ edges joined end to end. Here is a picture of $I_5$. 
\begin{center}
\includegraphics[width=3in]{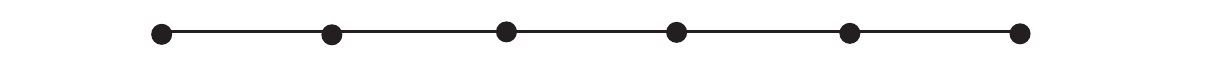}
\end{center}
Number the vertices left to right $1,\ldots, n+1$. Create a sequential vector field on $\Theta(I_n)$ as follows. First form all legal pairs of simplices $(\sigma,\sigma\cup\{1\})$. This leaves the singleton simplex $\{1\}$ unpaired, as well as all simplices which only avoid the edge between $1$ and $2$. These can be pictured thus:
\begin{center}
\includegraphics[width=3in]{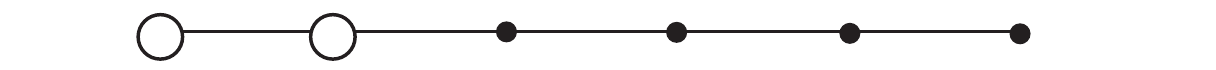}
\end{center}
Here the open circles indicate that those vertices are missing from the simplex. But now we know that the vertex $3$ must be in the simplex since otherwise the edge $\{2,3\}$ would be avoided. This we denote with a filled-in circle.
\begin{center}
\includegraphics[width=3in]{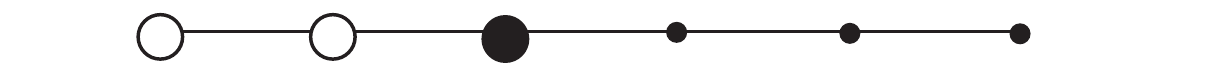}
\end{center}  
Now amongst these simplices, we form all legal pairs $(\sigma,\sigma\cup\{4\})$. Notice that if $\sigma$ is a simplex left over from the $1$ pairing (except $\{1\}$, which we leave alone for the rest of the calculation), and it doesn't contain $4$, then $\sigma\cup\{4\}$ is again a simplex of the same form: it avoids only the edge containing $1$. On the other hand, if $\tau$ contains $4$ but not $5$, then $\tau\setminus\{4\}$ avoids the edge $(4,5)$, and so was already paired at the first step. So the simplices unpaired after this second stage are of the form:
\begin{center}
\includegraphics[width=3in]{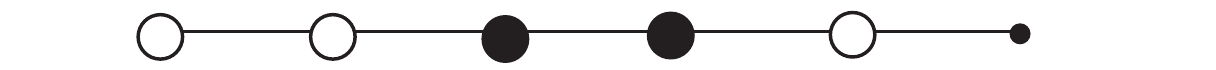}
\end{center}  
and again, the open vertex at $5$ implies the vertex at $6$ must be in the simplex.
\begin{center}
\includegraphics[width=3in]{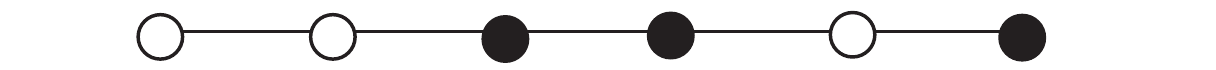}
\end{center} 
Our pictured example is now done. There is one critical simplex of dimension $2$ as pictured (with vertices \{3,4,6\}) together with the critical simplex $\{1\}$. Thus, using Theorem \ref{stuff}, $\Theta(I_5)\simeq S^2$. In general, continue this process, constructing the sequential vector field $\mathbf D_{1,4,7,\ldots,3m+1}$ where $m$ is the largest integer such that $3m+1\leq n+1$. 
 There are three cases depending on the congruence class of $n$ modulo $3$. If $n$ is divisible by $3$, then the end of the interval will look like this at the penultimate stage:
\begin{center}
\includegraphics[width=2in]{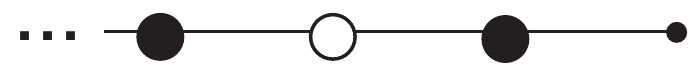}
\end{center} 
The last step will pair the two possible simplices together, demonstrating that $\Theta(I_n)$ is contractible. In the other two cases exactly one simplex will be left over. The exact formula is as follows:
 $$\Theta(I_n)\simeq\begin{cases}\bullet & n=3k\\
 S^{2k-1}& n=3k+1\\
 S^{2k}& n=3k+2\\
 \end{cases}$$
\end{example}

We move on to a slightly more complicated example.

\begin{example}
Let $P_n$ be the graph which is an $n$-sided polygon. For example, consider $P_9$, with vertices numbered cyclically around the polygon. Now create the vector field with all possible vectors
$(\sigma,\sigma\cup\{1\})$. The unpaired simplices are $\{1\}$ and those which only avoid an edge containing 1. Thus there are three possibilities:
\begin{center}
\includegraphics[width=3in]{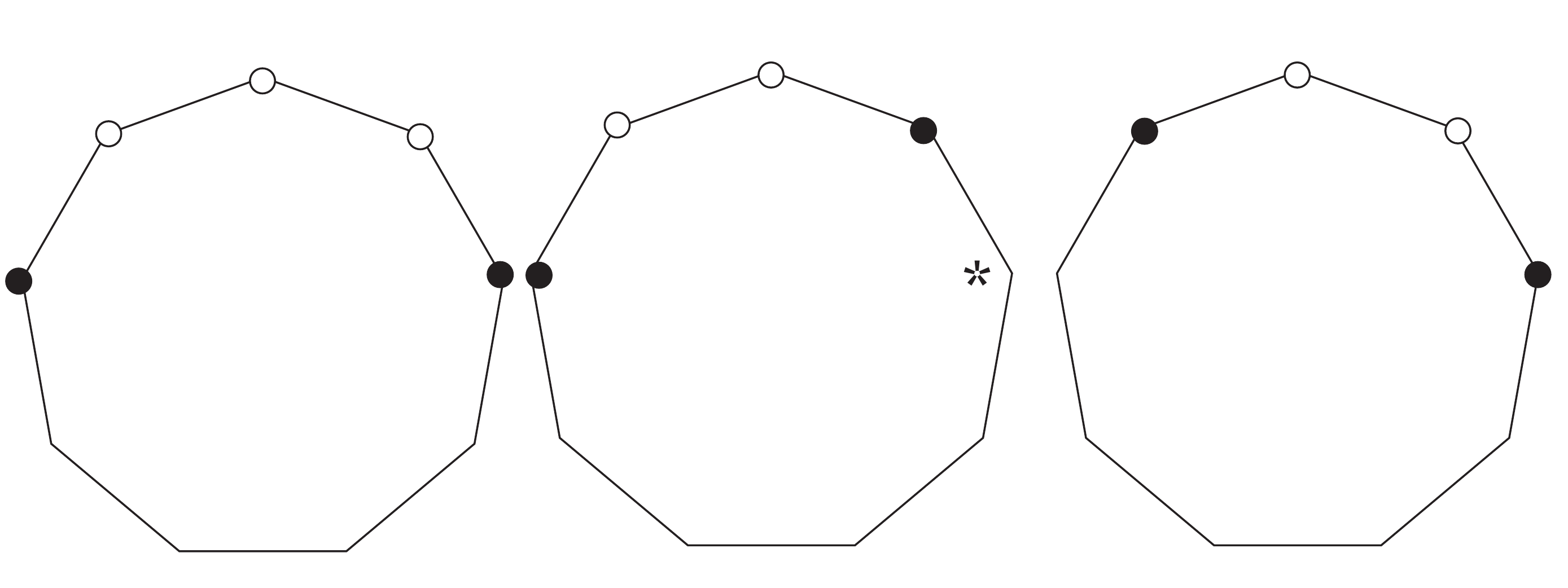}
\end{center} 
Now continue forming the sequential vector field by considering the starred vertex. This won't affect the two other pictured cases, so we get
\begin{center}
\includegraphics[width=3in]{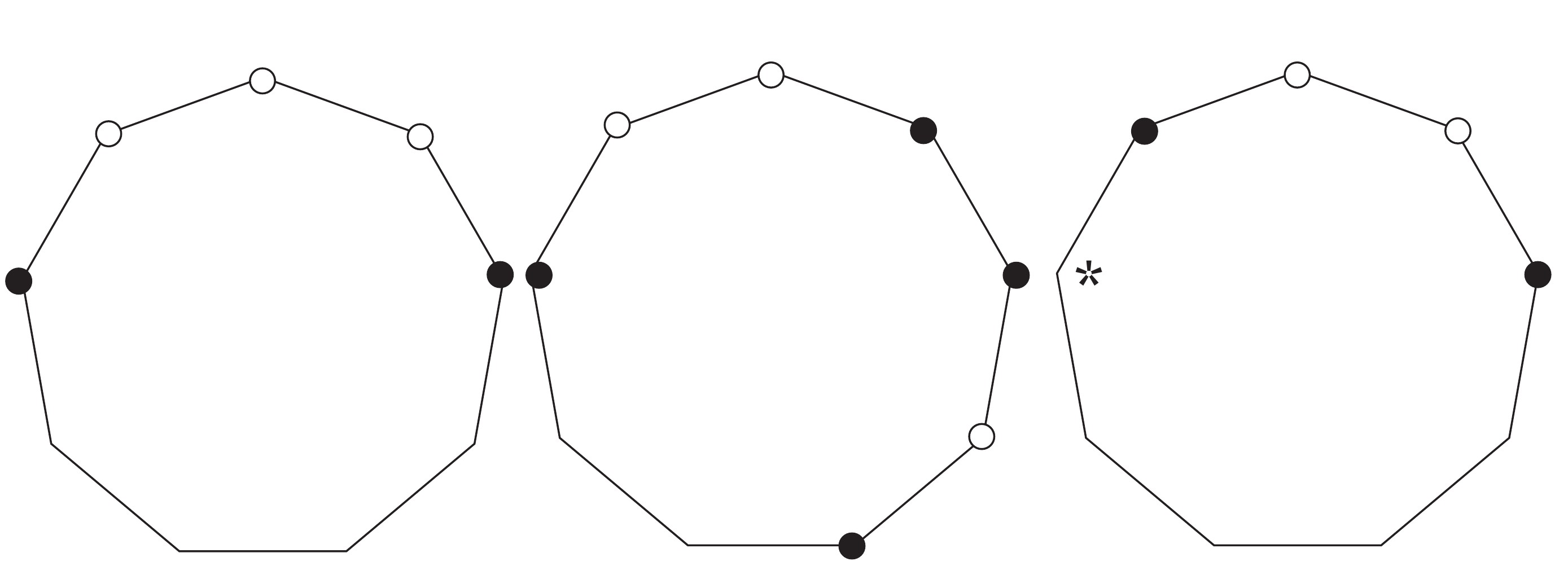}
\end{center} 
Repeating, with the indicated vertex:
\begin{center}
\includegraphics[width=3in]{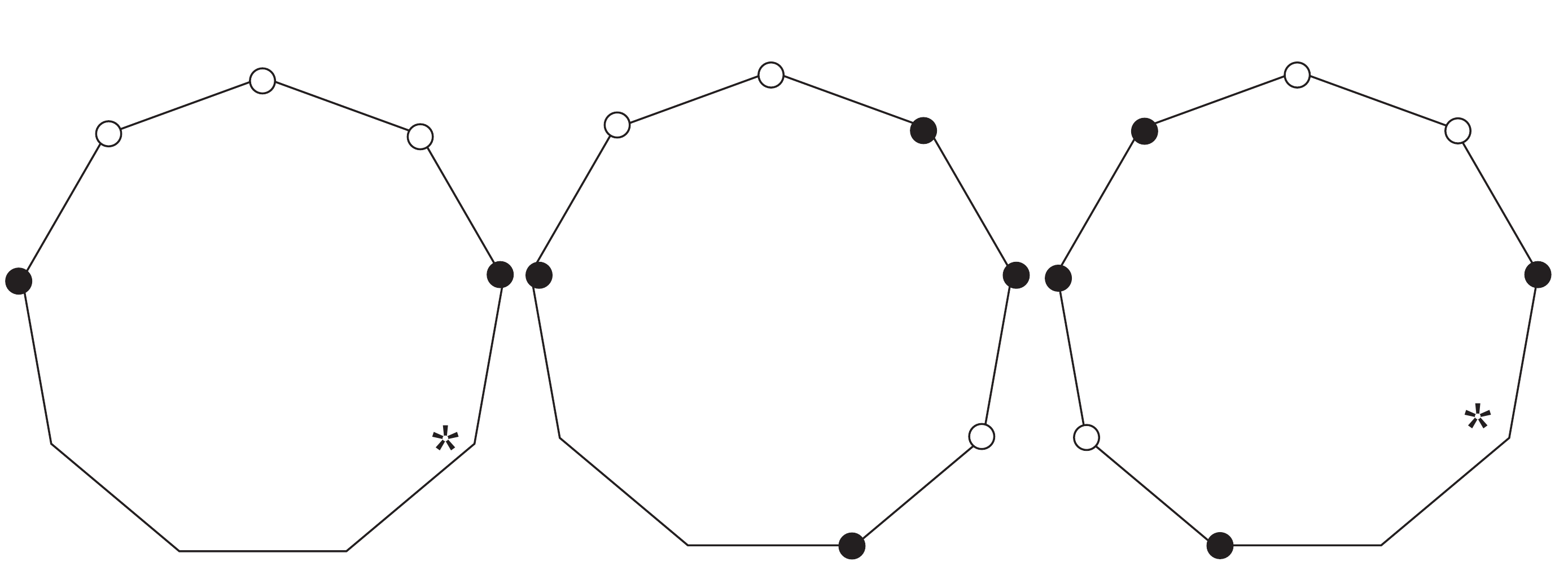}
\end{center} 
and then
\begin{center}
\includegraphics[width=3in]{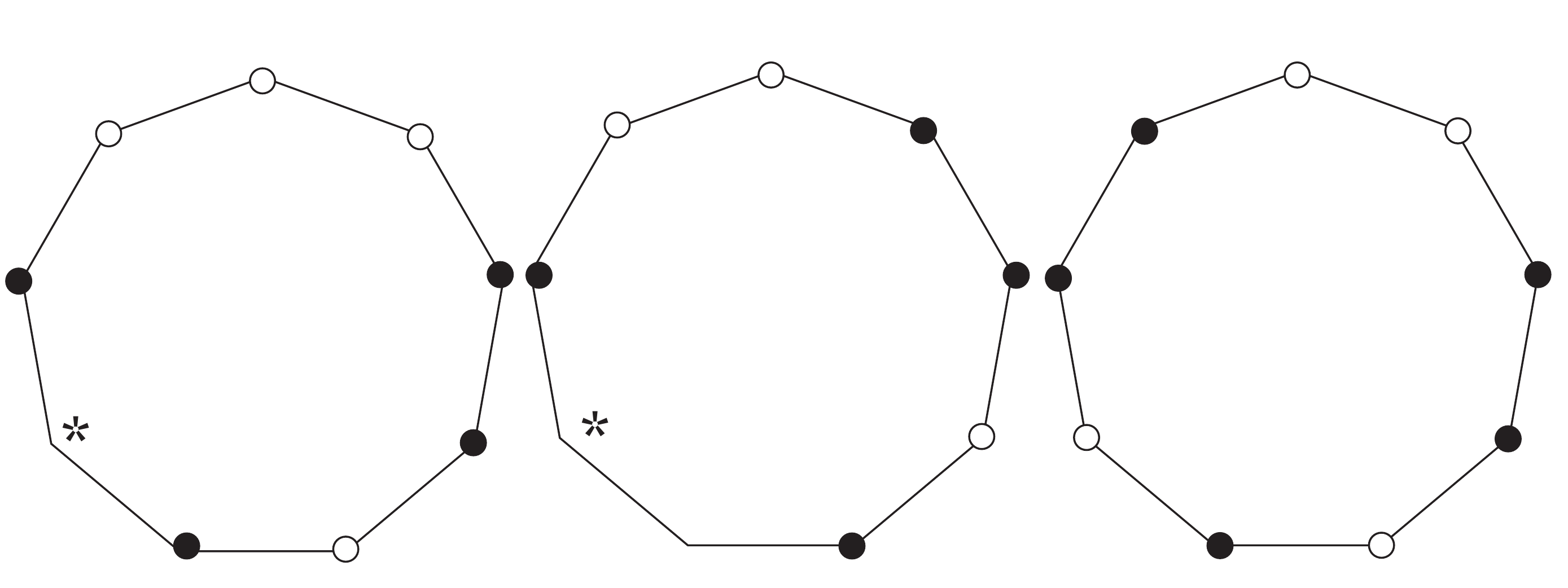}
\end{center} 
and
\begin{center}
\includegraphics[width=3in]{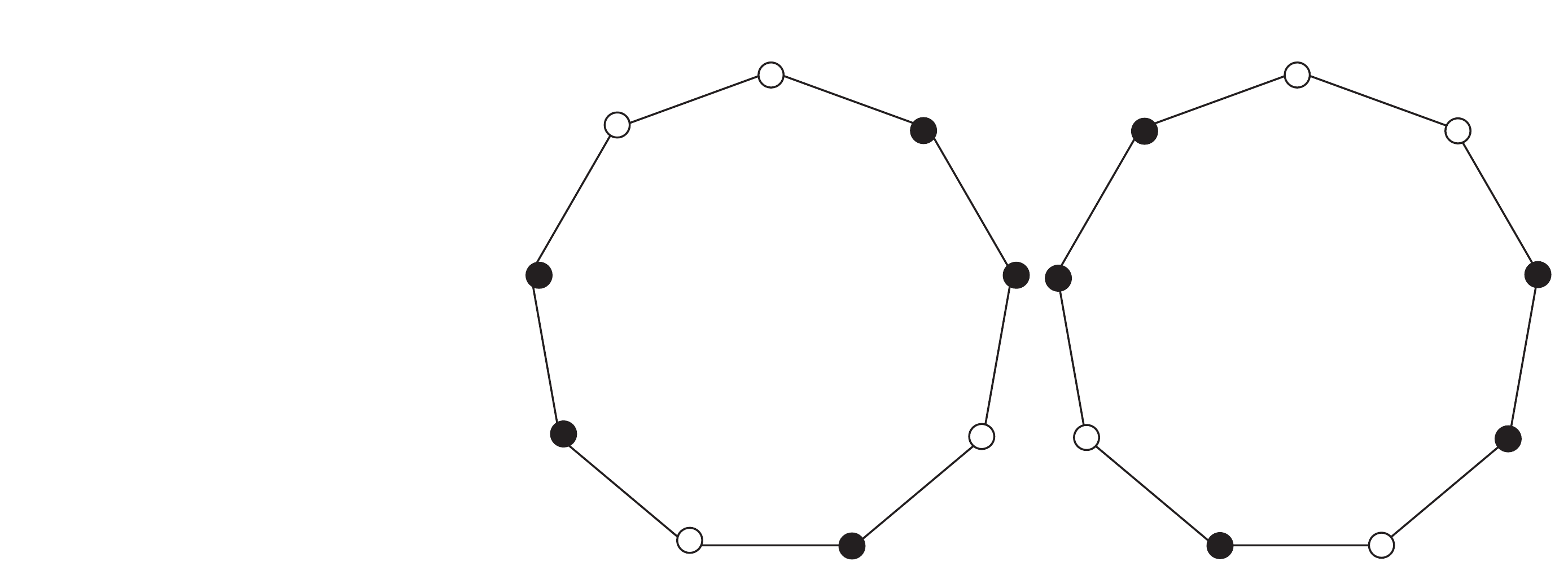}
\end{center} 
So we are left with two critical $4$-simplices, giving $\Theta(P_9)\simeq S^4\vee S^4$.

In general, we have $\Theta(P_n)\simeq \begin{cases}
S^{2k-2}\vee S^{2k-2} & n=3k\\
S^{2k-1} & n=3k+1\\
S^{2k-1} & n=3k+2
\end{cases}$
\end{example}

These examples exhibits a $3$-fold periodicity, and in fact,

\begin{proposition}\label{per}
Suppose a graph $\widetilde{G}$ is obtained from a graph $G$ by adding three interior vertices to an existing edge. Then $$\Theta(\widetilde{G})\simeq \Sigma^2 \Theta(G)$$
\end{proposition}
\begin{proof}
 One could construct vector fields on each of $\Theta(G)$ and $\Theta(\widetilde{G})$ which have a bijective correspondence between their critical simplices, such that the dimension of the $\widetilde{G}$ simplices is $2$ greater than the corresponding simplices for $G$ (excepting the unique $0$ simplex). While this could possibly be turned into a complete proof by analyzing the way the critical simplices attach to each other after crushing the simplices in the vector field, it is probably simpler to give a non-vector analysis in this case. 

Suppose the original edge has vertices $v,w$ and the subdivided edge has vertices $v$, $x_1$, $x_2$, $x_3$, and $w$ in that order. Let $B$ be the simplex spanned by the vertices of $G$. Then $\Theta(G)\subset B$. Let $\mathcal O_v\subset B$ be the subcomplex of simplices avoiding $v$ and let $\mathcal O_w$ be the subcomplex of simplices avoiding $w$. 
Then we have
\begin{align*}
\Theta(\widetilde{G})=&(\Theta(G)\setminus \operatorname{star}(\mathcal O_v\cap\mathcal O_w))*[x_1,x_2,x_3]\cup\\
&B*[x_1]\cup\\
&B*[x_3]\cup\\
&(\Theta(G)\cup\mathcal O_v\cup\mathcal O_w)*[x_2]\cup\\
&(\Theta(G)\cup\mathcal O_w)*[x_1,x_2]\cup\\
&(\Theta(G)\cup\mathcal O_v)*[x_2,x_3]
\end{align*}
This formula follows through a case analysis. If a simplex of $\Theta(\widetilde{G})$ contains both $v$ and $w$, but it avoids some edge in the original graph $G$, then the vertices $x_1,x_2,x_3$ can be freely added. This is the $(\Theta(G)\setminus \operatorname{star}(\mathcal O_v\cap\mathcal O_w))*[x_1,x_2,x_3]$ component above. (Recall $\operatorname{star}(\mathcal O_v\cap\mathcal O_w)$ is the union of the interiors of all simplices that have a face in the subcomplex $\mathcal O_v\cap\mathcal O_w$.) If, on the other hand, the simplex contains $v$ and $w$ but hits every edge of $G$, then it must omit $x_1$ or $x_3$, putting us in the second two terms of the above union. If it omits both $v$ and $w$, we can add $x_2$ freely, which is the fourth case. If it omits just $w$ then we can freely add $x_1$ and $x_2$ and still omit the edge from $x_3$ to $w$, giving us the penultimate term in the above union. The last term corresponds to just omitting $v$ and not $w$.

Note that each $y\in \Theta(G)$ is joined with one of the following contractible subsets of $[x_1,x_2,x_3]$:  
$$
[x_1,x_2,x_3], [x_1,x_2]\cup[x_2,x_3],[x_1,x_2],[x_2,x_3], \{x_2\}
$$
On the other hand each $y\in B\setminus \Theta(G)$ is joined to two distinct contractible subsets.
More specifically, each point in $B\setminus (\Theta(G)\cup\mathcal O_v\cup\mathcal O_w)$ is joined to $\{x_1,x_3\}$, each point in 
$\mathcal O_v\setminus \Theta(G)$ is joined to $[x_2,x_3]\cup\{x_1\}$ and each point in $\mathcal O_w\setminus \Theta(G)$ is joined to $[x_1,x_2]\cup\{x_3\}$.
Thus, if we shrink $[x_1,x_2,x_3]$ to a point $[x]$, this can be modeled by 
joining each point of $\Theta(G)$ to $[x]$ by a single line, and joining the points in $B\setminus \Theta(G)$ to $x$ by two lines, topologized so that these two lines get identified when you move to the subcomplex $\Theta(G)$. Now contract $B$. This yields a cell complex similar to the suspension of $B$: $B*\{\alpha,\beta\}$ except that the lines connecting $B$ to the two extra vertices are doubled away from $\Theta(G)$. 
These two lines can be regarded as coming from two separate copies of $B$ (called $B$ and $B'$), glued along $\Theta(G)$
so that $\Theta(\widetilde{G})\simeq \Sigma (B\cup_{\Theta(G)} B')$. (One copy of $B$ yields one set of lines, and the other copy $B'$ yields the second set.)
As in the proof of Proposition \ref{whichone}, we may assume that $B=C(\Theta(G))$, so that $B\cup_{\Theta(G)} B'=\Sigma(\Theta(G))$.
\end{proof}

Proposition \ref{per} is the exception rather than the rule when it comes to graph operations. Most simple graph operations do not have well-defined effects on the homotopy type of the theta complex. Indeed subdividing an edge by adding a single vertex will have wildly unpredictable effects on the homotopy type, as will connecting disjoint graphs by an edge.

Finally we move on to a cubic example. We use the notation $\vee_k X$ to denote a $k$-fold wedge of copies of $X$, which is to say $k$ copies of $X$ identified at a point.

\begin{example}\label{cubeexample}
We calculate $\Theta(\operatorname{Cube}(3,1))\simeq \vee_3 S^4$ using a sequential vector field. The leftover simplices after the first step will only omit edges incident to the first vertex. These can be sorted into three cases as follows, where the first vertex is the one in the lower left-hand corner.
\begin{center}
\includegraphics[width=3in]{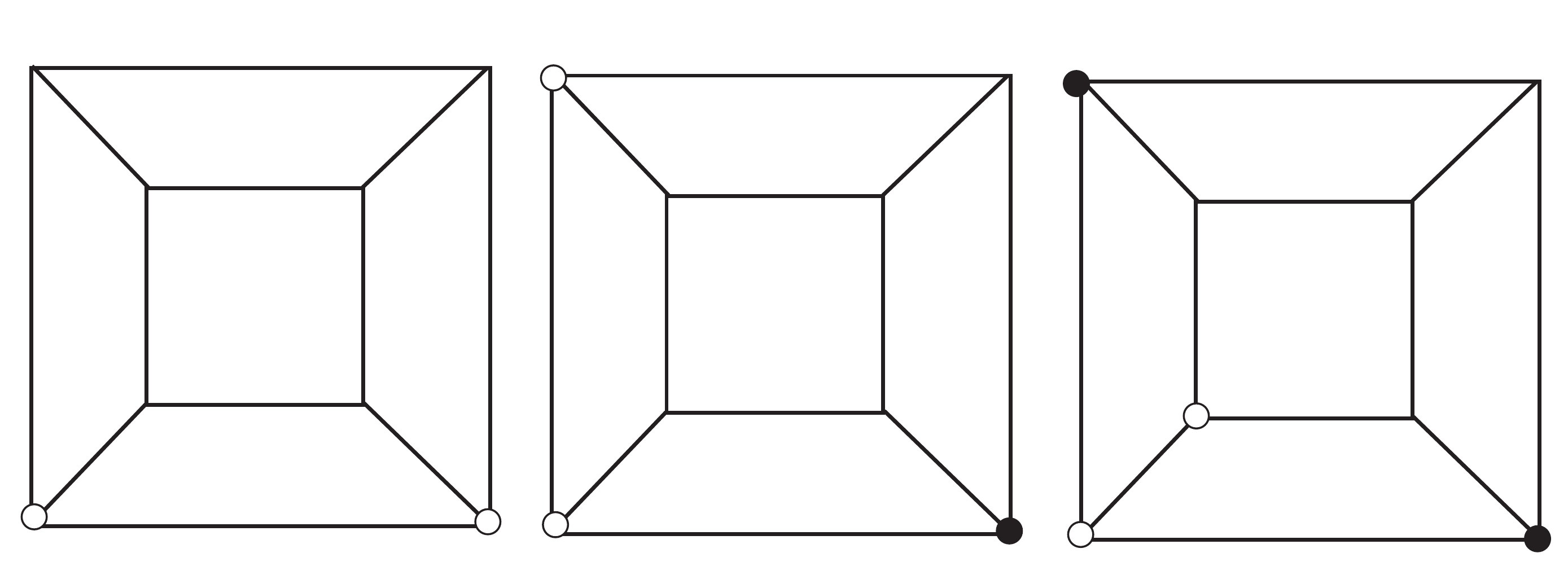}
\end{center} 
This forces some vertices to be in these critical simplices:
\begin{center}
\includegraphics[width=3in]{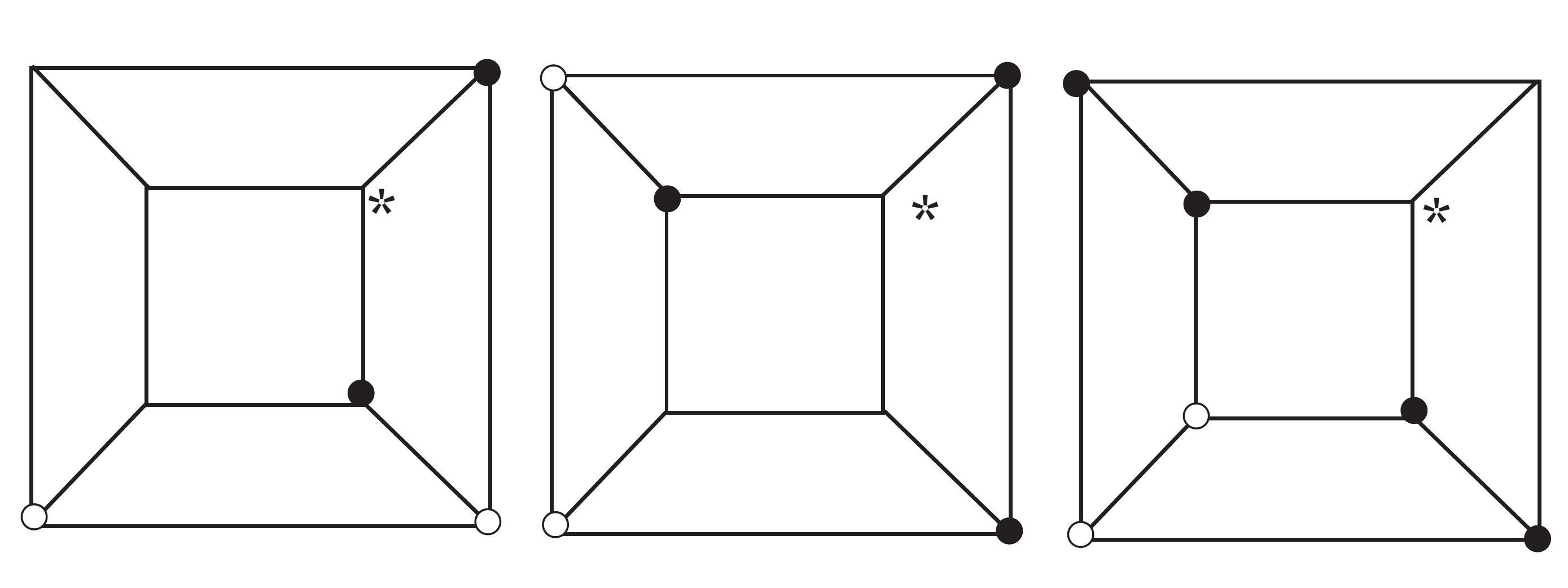}
\end{center} 
Continuing with the starred vertex, we get
\begin{center}
\includegraphics[width=3in]{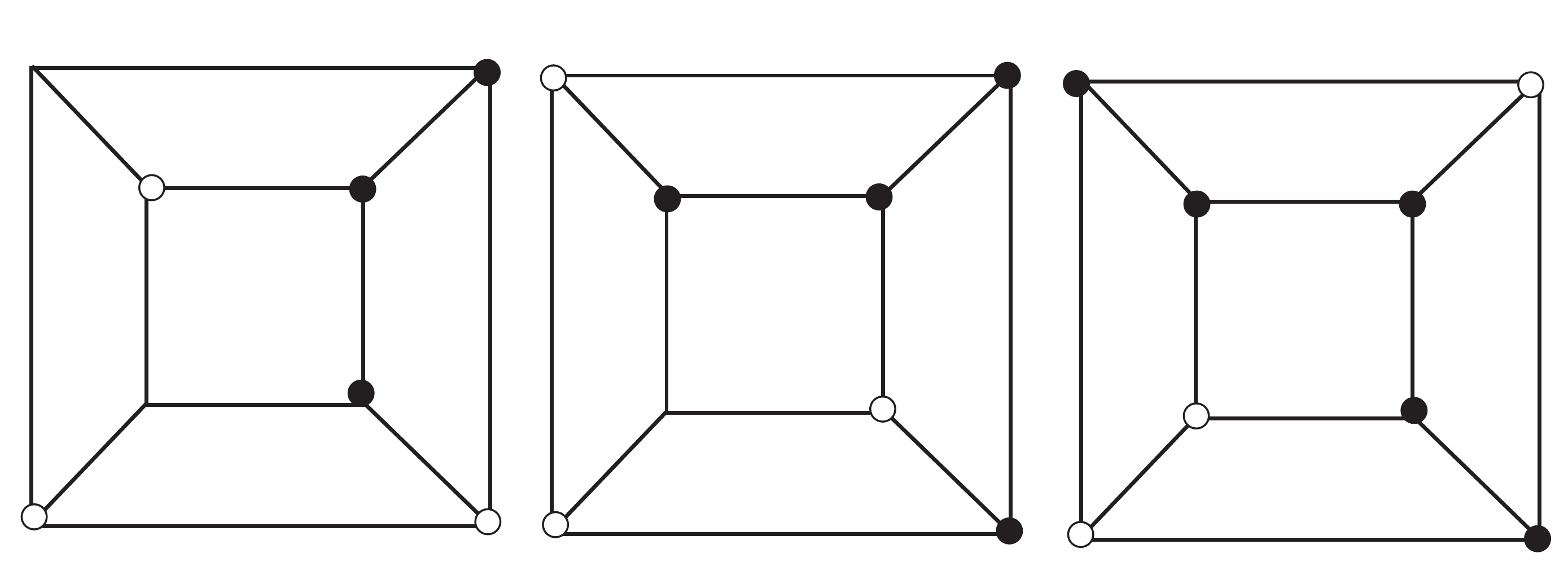}
\end{center}
This forces the final vertices to be filled in:
\begin{center}
\includegraphics[width=3in]{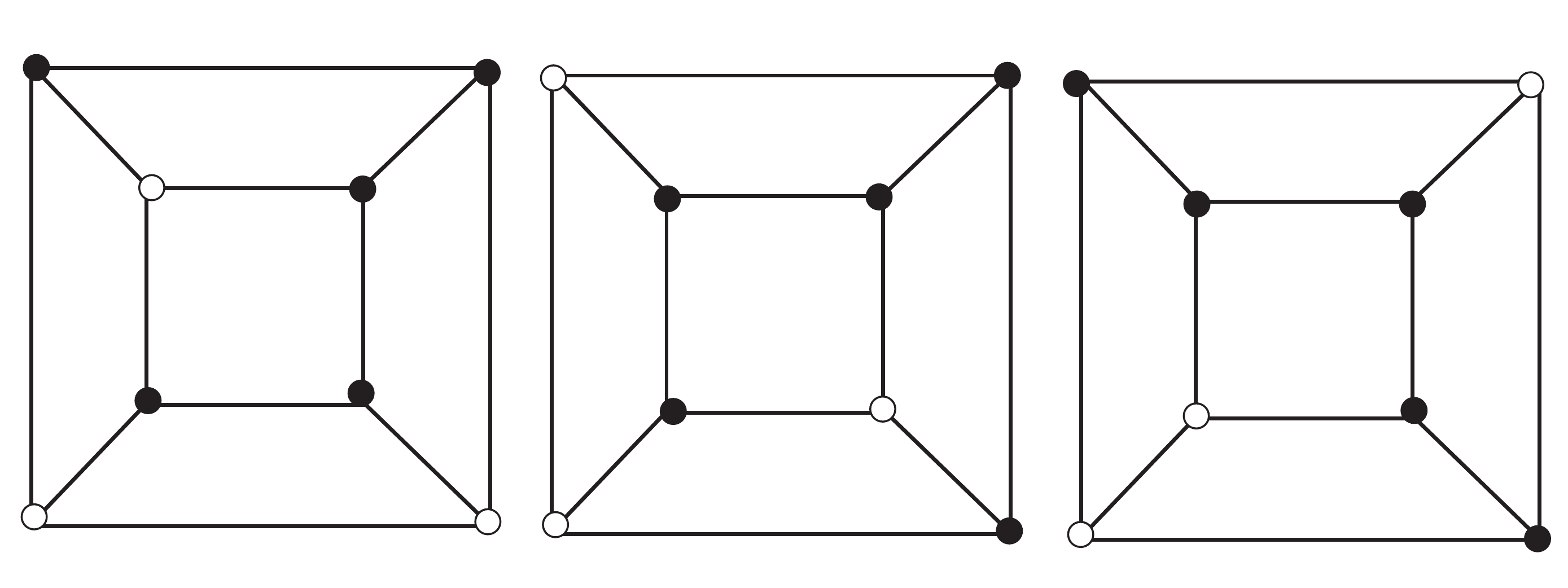}
\end{center}
Thus we get three critical $4$-simplices, implying $\Theta(\operatorname{Cube}(3,1))\simeq \vee_3 S^4$, as claimed.
\end{example}

\subsection{Cubes}\label{cubesec}
In this section, we collect some results about $\Theta(\operatorname{Cube}(n,k))$. Two cases are easy:

\begin{proposition}
The following statements are true.
\begin{enumerate}
\item $\Theta(\operatorname{Cube}(n,n-1))\simeq S^{n-1}$
\item $\Theta(\operatorname{Cube}(n,0))\cong S^{2^n-2}$
\end{enumerate}
\end{proposition}
\begin{proof}
Note that the complements of hyperedges of $\operatorname{Cube}(n,n-1)$ are also hyperedges
 and that neighborhoods of the codimension $1$ faces of a cube form a good cover of the cube's boundary $S^{n-1}$. Clearly, the cover of $\Theta(\operatorname{Cube}(n,n-1))$ by top dimensional simplices has the same nerve as this good cover.  So by the nerve theorem, $\Theta(\operatorname{Cube}(n,n-1))\simeq S^{n-1}$.
 
$\Theta(\operatorname{Cube}(n,0))$ consists of all proper subsets of the vertex set of the $n$-cube. This is the boundary of a simplex with $2^{n}$ vertices, which is a sphere of dimension $2^n-2.$
\end{proof}

We now present the results of computer calculations, both of the homotopy type of $\Theta(\operatorname{Cube}(n,k))$, and also of the reduced Euler characteristic, which we were able to determine for a slightly larger class of examples.

\begin{theorem}
\label{cube1}
The homotopy types of some examples of $\Theta(\operatorname{Cube}(n,k))$ are given in the following chart.
$$\begin{array}{c|cccccc}
n&k=0&k=1&k=2&k=3&k=4\\
\hline
1&S^0&&&&\\
2&S^2&S^1&&&\\
3&S^6&\vee_3S^4&S^2&&\\
4&S^{14}&\vee_7 S^{10}&\vee_{15}S^6&S^3&\\
5& S^{30}&?&?&\vee_{105}S^8&S^4
\end{array}
$$
In addition, both $\Theta(\operatorname{Cube}(5,1))$ and $\Theta(\operatorname{Cube}(5,2))$ are not homotopy equivalent to wedges of same-dimensional spheres.
\begin{enumerate}
\item The rational homology of $\Theta(\operatorname{Cube}(5,1))$ is trivial except  in the following cases:\\
 $H_0 \cong\mathbb Q$, $H_{22}\cong\mathbb Q^{10}$, $H_{24}\cong \mathbb Q$
 \item The rational homology of $\Theta(\operatorname{Cube}(5,2))$ is trivial except in the following cases:\\ $H_0 \cong \mathbb Q$, $H_{14} \cong \mathbb Q^{60}$, $ H_{15} \cong $ ?, $ H_{16} \cong $ ?, $H_{17} \cong $ ?, $H_{18} \cong \mathbb Q^{16} $. The question marked groups may or may not be trivial.
 \end{enumerate}
\end{theorem}
{\bf Note:} The integer homology groups of these examples are currently unknown.

\begin{theorem}
\label{cube2}
The reduced Euler characteristics of some examples of $\Theta(\operatorname{Cube}(n,k))$ are given in the following chart.
$$\begin{array}{c|ccccccc}
n&k=0&k=1&k=2&k=3&k=4&k=5&k=6\\
\hline
1&1&&&&&&\\
2&1&-1&&&&&\\
3&1&3&1&&&&\\
4&1&7&15&-1&&&\\
5&1&11&57&105&1&&\\
6&1&143&&&&-1&\\
7&1&7715&&&&&1\\
\end{array}
$$
\end{theorem}
The fact that the Euler characteristics are always even is proven in the section~\ref{groupsection}.

The following conjecture is consistent with the known data and with mod $p$ Euler characteristic calculations, as we will see in a later section (Theorem 9). Note that $n!!=n(n-2)(n-4)\cdots 1$, when $n$ is odd.

\begin{conjecture}
\label{factorialconjecture}
$\Theta(\operatorname{Cube}(n,n-2))\simeq \vee_{(2n-3)!!} S^{2n-2}$
\end{conjecture}

Indeed, bearing in mind the goal of distinguishing $2$-SAT from $k$-SAT, $k\geq 3$, the following conjecture could prove very useful.

\begin{conjecture}\label{spread}
$\Theta(\operatorname{Cube}(n,k))$ is not homotopy equivalent to a wedge of same-dimensional spheres for $0< k<n-2$ for $n$ sufficiently large. Indeed, the nontrivial homology groups span an increasing range of dimensions as $n$ increases.
\end{conjecture}

If these conjectures are true, they would show a dramatic difference in the homotopy types of 
$|2\text{-}SAT\text{-}n|\simeq  \Theta(\operatorname{Cube}(n,n-2))$ and $|k\text{-}SAT\text{-}n|$ for $k\geq 3$. 



\subsection{Other complexes}
The study of $\Theta(\mathcal H)$ is a fascinating area in its own right. In this section, we present calculations for some hypergraphs besides cubes. Since cubes are an example of a regular polytope, it might be natural to wonder what happens for other regular polytopes. Besides cubes, there are two other infinite classes of polytopes: simplices (generalized tetrahedra) and cross polytopes (generalized octahedra).

\begin{definition}
\item
\begin{enumerate}
\item Let $\operatorname{Simp}(n,k)$ denote the hypergraph whose vertices are the vertices of the $n$-simplex and whose hyperedges arise from the $k$-faces of the simplex.
\item Define the $n$-dimensional cross-polytope to be the simplicial complex which is the iterated suspension $\Sigma^{n-1} S^0$. Define the hypergraph $\operatorname{CrossPoly}(n,k)$ to have the same vertex set as $\Sigma^{n-1} S^0$ and to have a hyperedge for every $k$ dimensional face.
\end{enumerate}
\end{definition}

\begin{theorem}
The following homotopy equivalences hold.
\begin{enumerate} 
\item $\displaystyle\Theta(\operatorname{Simp}(n,k))\simeq \vee_{\binom{n}{n-k}} S^{n-k-1}$ 
\item $\Theta(\operatorname{CrossPoly}(n,k))\simeq \vee_{\binom{n-1}{k}}S^{2n-k-2}$
\end{enumerate}
\end{theorem}
\begin{proof}
Notice that $\Theta(\operatorname{Simp}(n,k))$ is the $(n-k-1)$-skeleton of the $n$-simplex. The homotopy type of this is easily calculated by shrinking the star of a vertex to a point, leaving a wedge of $(n-k-1)$-spheres, one for every $n-k-1$ face missing that vertex. 

$\operatorname{CrossPoly}(n,k)$ can be modeled as follows. Let the vertices be $v_1^+,v_1^-,\ldots,v_n^+, v_n^-$. A collection of $k+1$ vertices forms a $k$-face if and only if it does not contain both vertices in any pair $v_i^+,v_i^-$.

Now form the sequential vector field $\mathbf D_{v_1^+,\ldots,v_n^+}$.
 The critical simplices in $C_1$ are those which avoid only $k$-faces containing $v_1^+$. In particular, they must contain $v_1^-$ since otherwise any $k$-face that is avoided by $v_1^+$ could be converted to a $k$ face avoided by $v_1^-$ by replacing $v_1^+$ with $v_1^-$.  In each critical simplex, there must be a set of indices $I$ of size $k$ such that $1\not\in I$ and for every $i\in I$ at least one of $v_i^\pm$ is not in the simplex, and for every $j\not\in I\cup\{1\}$, both of $v_j^\pm$ are in the simplex.
  Now we calculate $C_2$. Evidently, all simplices in $C_1$  which contain both $v_2^\pm$ persist to $C_2$. The other elements of $C_1$ which remain unpaired and therefore persist to $C_2$
 are simplices that contain  $v_2^-$ but not $v_2^+$. Continuing, at the $\ell$th stage of the vector field's construction, if a simplex contains both $v_\ell^\pm$ then it is not paired, or if it contains $v_\ell^-$ but not $v_\ell^+$ it is not paired. In the end, the critical simplices are given by choosing $k$ indices from $2,\ldots, n$, filling in all vertices except $v_1^+$ and $v_i^+$ where $i$ is in the chosen set of $k$ indices. There are $2n-k-1$ vertices in such a configuration, corresponding to a $2n-k-2$-cell, and there are $\binom{n-1}{k}$ ways to choose the index set, giving the desired result. 
\end{proof}

In addition to the above infinite classes of regular polyhedra, in three dimensions we also have
the icosahedron and dodecahedron.
 Let $\operatorname{Dodec}(k)$ represent the hypergraph of $k$-dimensional faces of a dodecahedron, and $\operatorname{Icos}(k)$ represent the hypergraph of $k$-faces of an icosahedron.
\begin{proposition}\label{blah}
 The following homotopy equivalences hold
\begin{enumerate}
\item $\Theta(\operatorname{Dodec}(1))\simeq \vee_4 S^{12}$
\item $\Theta(\operatorname{Icos}(1))\simeq S^7\vee\vee_6 S^{8}$
\item $\Theta(\operatorname{Dodec}(2))\simeq \Sigma^3\mathbb{RP}^2$
\item $\Theta(\operatorname{Icos}(2))\simeq \Sigma^3\mathbb{RP}^2$
\end{enumerate}
\end{proposition}
These complexes were calculated using a mixture of computer and hand calculations. The computer program performed as many simple-homotopy reductions as it could find, leaving a small collection of simplices in each case. The final results were achieved by examining the way they attach to each other. Note that since $\operatorname{Dodec}(2)^*=\operatorname{Icos}(2)$, the equality of the last two is no accident.

Finally, the three additional four dimensional regular polytopes were too complex to analyze by computer.

\section{Group Actions}\label{groupsection}

Let $\mathcal H=(V,H)$ be a finite hypergraph, with a group action $G$. That is $G$ acts on the vertices and carries hyperedges to hyperedges. We define the quotient hypergraph,  $\mathcal H/G$, to have vertex set equal to $V/G$ and the hyperedges to be the images of the hyperedges under the quotient $V\to V/G$. 

\begin{theorem}\label{grouptheorem}
Consider a finite hypergraph $\mathcal H$, acted on by a $p$-group $G$. Then
$$\chi(\Theta(\mathcal H))\equiv \chi(\Theta(\mathcal H/G)) \mod p.$$
\end{theorem}
\begin{proof}
The group $G$ acts on the set of simplices of $\Theta(\mathcal H)$. By the index counting formula, the total number of simplices is equal to the sum of the indices of the stabilizers of orbit representatives. If a simplex is not stabilized by all of $G$, then the index is a multiple of $p$, so that such simplices can be discarded when counting modulo $p$. We then are left with counting simplices (subsets of vertices of $\mathcal H$) which are stabilized by the whole group $G$. These are in 1-1 correspondence with simplices in the quotient $\Theta((\mathcal H/G))$. If a $G$-stabilized simplex $\alpha$ in $\mathcal H$
omits some hyperedge $h$, then the quotient simplex $\bar{\alpha}$ omits $\bar{h}$, since if $g\cdot h\cap \alpha\neq \emptyset$ for some group element $g$, then $h\cap g^{-1}\cdot \alpha\neq\emptyset$, a contradiction since $g^{-1}\cdot\alpha =\alpha$. Similarly, if a set of vertices in the quotient avoids a quotient hyperedge $h$, then the union of $G$-orbits of these vertices will avoid any lift of $h$. If $p=2$ we are done
since the Euler characteristic modulo $2$ does not see the dimension of the simplices that it counts. If $p\neq 2$, we must show that the mod-2 dimension of a $G$-stabilized simplex is the same as the corresponding simplex on the quotient. This follows because the size of an orbit of any vertex under $G$ will be a power of $p$, which is odd.\end{proof}

\begin{corollary}
Suppose $\mathcal H$ has at least one hyperedge and a $p$-group acts transitively on the vertices. Then $\chi(\Theta(\mathcal H))\equiv 0\mod p$.
\end{corollary}
\begin{proof}
The quotient hypergraph is a single vertex and a single hyperedge. Thus $\Theta((\mathcal H/G))=\emptyset$, which has Euler characteristic $0$.
\end{proof}

This implies

\begin{corollary}
For every $k\leq n$, $\chi(\Theta(\operatorname{Cube}(n,k)))$ is even.
\end{corollary} 
\begin{proof}
The group $\mathbb Z_2^n$ acts on $\operatorname{Cube}(n,k))$, and is transitive on the vertices.
\end{proof}

Let's check another example.
\begin{example}
Let $\mathbb Z_5$ act on $\operatorname{Dodec(n)}$ by rotation through an axis piercing the center of a pentagonal face. Then $\operatorname{Dodec(1)}/\mathbb Z_5$ consists of four vertices $v_1,v_2,v_3,v_4$ with edges connecting $v_i$ to $v_{i+1}$ and with the singleton hyperedges $\{v_1\}$ and $\{v_4\}$. To calculate $\Theta((\operatorname{Dodec}(1)/\mathbb Z_5))$ we can throw away any hyperedges that contain existing hyperedges. Hence we only really have three hyperedges. Using the sequential vector field arising from the sequence $v_1,v_2$ we have only two critical simplices: $\{v_1\}$ and $\{v_2,v_3\}$, so we get a circle.   Thus $\chi(\Theta(\operatorname{Dodec}(1)))\equiv 0 \mod 5$. This meshes with the answer of $\chi=5$ coming from Proposition \ref{blah}. Similarly, the quotient of $\operatorname{Dodec}(2)$ is a hypergraph with the same $4$ vertices and with hyperedges 
$\{v_1\},\{v_4\},\{v_1,v_2,v_3\}$ and $\{v_2,v_3,v_4\}$. These latter two can be discarded. Since this hypergraph contains an isolated vertex it is contractible. Hence $\chi(\Theta(\operatorname{Dodec}(1)))\equiv 1 \mod 5$, which is also consistent with Proposition \ref{blah}.

\end{example}

Finally, we use $p$-groups to analyze cubes and give support to Conjecture \ref{factorialconjecture}.

\begin{theorem}\label{groupcube}
Let $p$ be an odd prime and $n\geq p$, then 
$$\chi(\Theta(\operatorname{Cube}(n,n-2)))\equiv 1 \mod p$$
\end{theorem}

To see this, let $\mathbb Z_p$ act on $\operatorname{Cube}(n,n-2)$ by considering the cube's vertices to be the set of subsets of $\{x_1,\ldots,x_{n-p},y_1,\ldots,y_p\}$ and letting $\mathbb Z_p$ cycle the $y_i$'s. Theorem \ref{groupcube} now follows from the following stronger theorem.

\begin{theorem} Let $\displaystyle\mathcal H=\operatorname{Cube}(n,n-2)/\mathbb Z_p$. Then $\Theta(\mathcal H)$ is contractible.
\end{theorem}
\begin{proof}
First, consider the case $n=p$. Given a monomial $m$, let $\lsem m\rsem$ denote the set of all submonomials, including $1$ and $m$. Then
$\operatorname{Cube}(p,p-2)$ has vertices in one-one correspondence with $\lsem y_1y_2\cdots y_p\rsem$ and has hyperedges of the form

\begin{enumerate}
\item \label{c} $\lsem y_1\ldots y_py_i^{-1}y_j^{-1}\rsem$
\item \label{d} $y_i\lsem y_1\ldots y_py_i^{-1}y_j^{-1}\rsem$
\item $y_iy_j \lsem y_1\ldots y_py_i^{-1}y_j^{-1}\rsem$
\end{enumerate}

Then the vertices of $\mathcal H$ are $p$-\emph{necklaces}, that is monomials in the variables $y_1,\ldots,y_p$ considered up to cyclic symmetry. (The reason for this terminology will become apparent in the next paragraph and is also illustrated in Figure \ref{hyperexample}.)
 A necklace which is an equivalence class of a monomial $m$, will be denoted by $\overline{m}$.
The degree of a necklace is defined to be the degree of the monomial. (So, for example the $5$-necklaces of degree $2$ are $\overline{y_1y_2}$ and $\overline{y_1y_3}$.) 
The hyperedges of $\mathcal H$ are induced by the hyperedges in the above list. 
  
Now we claim every hyperedge containing the necklace $\bar{1}$ also contains the necklaces of degree $\leq (p-1)/2$. The only hyperedges that contain $\bar{1}$  are of type (\ref{c}) in the above list. Thus, this amounts to showing that every $p$-necklace of degree $p-2$ contains every $p$-necklace of degree $\leq (p-1)/2$. Visualize a necklace as a circle of white and black beads, with black beads indicating the presence of a variable and white beads indicating its absence. The $\mathbb Z_p$ action is by rotation, so these pictures should be considered up to rotational symmetry.
In this language, a degree $p-2$ necklace will have exactly two white beads. Visualize these connected by a chord, say of length $a$.
Then we need to show that there is a chord of any such possible length $a$ between two white beads of a necklace with $(p-1)/2$ black beads. There are $p$ chords of length $a$, and each bead is in $2$ such chords. Thus the $(p-1)/2$ black beads can hit at most $2(p-1)/2$ of the chords of length $a$, leaving at least one chord between white beads.

\begin{figure}
\begin{center}
\includegraphics[width=.8\linewidth]{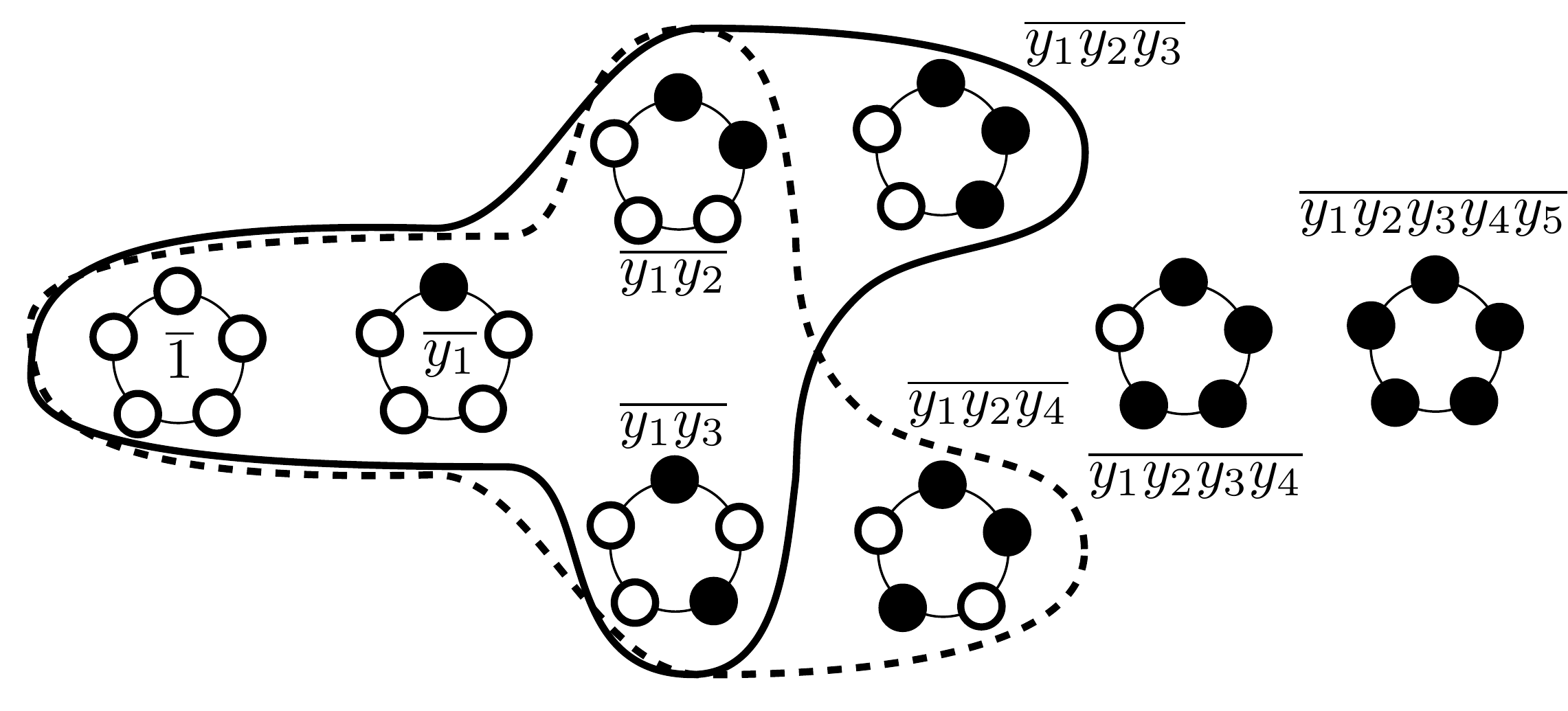}
\end{center}
\caption{A partial picture of $\operatorname{Cube}(5,3)/\mathbb Z_5$. The necklaces represent equivalence classes of monomials up to the cyclic $\mathbb Z_5$ action. The two hyperedges containing $\overline{1}$ are drawn: the quotients of $\lsem y_1y_2y_3\rsem$ and $\lsem y_1y_2y_4\rsem$, and they do indeed contain the three vertices nearby $\overline{1}$.}\label{hyperexample}
\end{figure}

We also claim that every hyperedge containing $\overline{y_1\cdots y_p}$ hits every necklace of degree greater than $(p-1)/2$. This follows because there is a $\mathbb Z_2$ action on $\mathcal H$ obtained by sending $y_i^{a_i}$ to $y_i^{1-a_i}$. One can check this
 by noting that the action exists on $\operatorname{Cube}(n,n-2)$ and is compatible with the $\mathbb Z_p$ action. Thus the two vertices $\bar{1}$ and $\overline{y_1\cdots y_p}$ can be interchanged, and the above argument applied. 

Let $\bar{1}$ and $\overline{y_1\ldots y_p}$ be called \emph{end} vertices. Every necklace except $\bar{1}$ of degree $\leq (p-1)/2$ will be said to be \emph{nearby} $\bar{1}$, and every necklace except $\overline{y_1\ldots y_p}$ of degree $> (p-1)/2$ will be said to be \emph{nearby} $\overline{y_1\ldots y_p}$.  In Figure~\ref{hyperexample}, the vertices on the left side of the diagram are all nearby $\overline{1}$, and the ones on the right are all nearby $\overline{y_1y_2y_3y_4y_5}$.

Now suppose a simplex of $\Theta(\mathcal H)$ contains some necklace nearby $\bar{1}$. Then, because every hyperedge meeting $\bar{1}$ also meets this necklace, $\bar{1}$ can be added  or removed and we would still have a legal simplex. (With the exception of the singleton simplex $[\bar 1]$.) Pair all simplices containing a nearby vertex to $\bar{1}$ into vectors of the form $(\sigma,\sigma\cup\{\bar{1}\})$.
The critical simplices are those which do not contain any vertex nearby $\bar{1}$. 
Repeat this procedure for the vertex $\overline{y_1\cdots y_p}$, yielding at most three critical simplices
$[\bar{1}],[\overline{y_1\cdots y_p}],$ and $[\bar{1},\overline{y_1\cdots y_p}] $.
These three simplices miss hyperedges of type (\ref{d}) above. Thus they are each legal, and we can, for instance pair the second two together, leaving a single critical $0$ simplex.  (This is where the argument fails for $\operatorname{Cube}(n,n-1)$.)

One must check this is a gradient vector field. Note that a gradient path consists of two alternating operations: removing a vertex from a simplex, and adding a vertex to a simplex, with the proviso that adding a vertex must correspond to a vector. So suppose we have a gradient path, and
a necklace other than $\bar{1}$ or $\overline{y_1\cdots y_p}$ is removed at some stage.
This can never be added back in, since such necklaces are not added in by any vector. Thus the gradient path cannot be a loop.
 So suppose the gradient path only has removal of the vertices $\bar{1}$ or $\overline{y_1\cdots y_p}.$ 
 Suppose it starts $\sigma_0,\tau_0,\sigma_1,\cdots$.Suppose that $\tau_0=\sigma_0\cup\{\bar{1}\}$. Then since we can't remove a vertex we just added, $\sigma_1$ is forced to be $\tau_1\setminus\{\overline{y_1\cdots y_p}\}$. But now $\sigma_1$ is not the first coordinate of any vector, so the path terminates and is not a loop.

Now we consider the general case of $\operatorname{Cube}(n,n-2)$. Note that $\mathbb Z_2^{\oplus n}$ acts on this cube, and that $G= \mathbb Z_2^{\oplus n-p}\oplus \mathbb Z_2$ acts on the quotient, with the first $n-p$ $\mathbb Z_2$'s flipping the parity of the $x_i$'s and the last one working on all of the $y_i$'s simultaneously. The set of vertices of $\mathcal H$ is thus equal to  $\mathbb Z_2^{\oplus n-p}\cdot \overline{\lsem y_1\ldots y_p\rsem}$, where $\overline{\lsem y_1\ldots y_p\rsem}$ represents the quotient of the hyperedge $\lsem y_1\ldots y_p\rsem$.

 This quotient $\overline{\lsem y_1\ldots y_p\rsem }$ has two distinguished vertices $\overline{1}$ and $\overline{y_1\ldots y_p}$, which we call end vertices, as before. Also as before, a vertex is said to be nearby the $\bar{1}$ vertex if it represents a necklace of degree $\leq (p-1)/2$. In general, a vertex of $\mathcal H$ is said to be an end vertex if it is in the $G$-orbit of an end vertex, and a vertex $v$  is said to be nearby an end vertex $w$ if
 $g\cdot v$ is nearby $g\cdot w=\bar{1}$, for some $g\in G$. We claim that every hyperedge containing an end vertex also contains each nearby vertex. It suffices to consider the end vertex $\bar{1}$. The hyperedges containing $\bar{1}$ are quotients of hyperedges of the form $\lsem x_1\cdots x_{n-p}y_1\ldots y_py_i^{-1}y_j^{-1}\rsem $, $\lsem x_1\cdots x_{n-p}y_1\ldots y_px_i^{-1}y_j^{-1}\rsem $ and $\lsem x_1\cdots x_{n-p}y_1\ldots y_px_i^{-1}x_j^{-1}\rsem $. If we look at the intersection of these edges with $\lsem y_1\cdots y_p\rsem $ we get $\lsem y_1\cdots y_py_i^{-1}y_j^{-1}\rsem $, $\lsem y_1\cdots y_p y_j^{-1}\rsem $ and $\lsem y_1\cdots y_p\rsem $. We have already seen when we argued the $n=p$ case that this first type must hit all vertices nearby $\bar{1}$, and the other two types are even larger. 

Now enumerate the end vertices in some fashion, say beginning with $\bar{1}.$ We create a vector field, by first pairing together all simplices which contain a vertex nearby $\bar{1}$ by vectors $(\sigma,\sigma\cup\{\bar{1}\})$.
The critical simplices are exactly those which do not contain any of the vertices nearby to $\bar{1}$. Now continue with the next end vertex, and proceed through all the end vertices. As in the $n=p$ case, which had two end vertices, we are left with simplices which are subsets of the end vertices. Note that the quotient of the hyperedge $y_i\lsem x_1\ldots x_{n-p}y_1\ldots y_py_i^{-1}y_j^{-1}\rsem$ does not contain any end vertices. Thus there is a critical simplex for every nonempty subset of the end vertices.
Form vectors of all legal pairs $(\sigma,\sigma\cup\{\bar 1\})$ among these, yielding a single critical $0$ simplex $[\bar 1]$, as in the $n=p$ case.

Now we argue that this is a gradient vector field. Consider a gradient path. As before if we ever remove a non-end vertex, we can never regain it. Hence we can only remove end vertices. Suppose that a simplex avoids all nearby vertices to ends $1$ to $k-1$, but that it contains a nearby vertex to the $k$th end. Call such a simplex $k$-deficient. By definition, every $k$-deficient simplex is part of a vector toggling the $k$th end vertex. Now suppose we have a gradient path, starting with a $k$ deficient 
simplex $\sigma_0$. Then $\tau_0$ is formed by adding the $k$th end vertex. $\sigma_1$ is formed by removing some other end vertex. But now $\sigma_1$ is still $k$-deficient, which means it is the right coordinate of a vector which deletes the $k$th end.  Thus the gradient path cannot continue, and is certainly not a loop.
\end{proof}


\section{Estimating the connectivity of theta complexes}
In this section we consider the case of $\operatorname{Cube}(n,1)$, which are actual graphs and not hypergraphs. This allows us to apply a connectivity estimate of Engstr\"om \cite{e1} to the Alexander dual independence complex.
\begin{theorem}
If $G$ is a graph with $m$ vertices and maximal valence $d$, then $I(G)$ is $\lfloor(m-1)/2d-1\rfloor$-connected.
\end{theorem}
For $\operatorname{Cube}(n,1)$, we have $m=2^n$ and $d=n$, so according to Engstr\"om's theorem we know $I(G)$ is $(2^n-1)/(2n)-1$-connected. The actual connectivity for $n=2,3,4,5$ is $-1,0,2,4$ whereas this estimate yields $-1,0,0,2$, and so is not sharp in general.

\end{document}